\documentclass[12pt,reqno]{amsart}
\usepackage{amsmath, amsthm, amssymb}

\usepackage{comment}
\usepackage{color}

\advance \topmargin by -\headheight
\advance \topmargin by -\headsep
     
\evensidemargin \oddsidemargin
\newtheorem{theorem}{Theorem}[section]
\newtheorem{lemma}[theorem]{Lemma}
\newtheorem{corollary}[theorem]{Corollary}

\newtheorem{hypothesis}[theorem]{Hypothesis}

\theoremstyle{definition}

\theoremstyle{remark}

\numberwithin{equation}{section}

\def\bfb{{\mathbf b}}

 \def\bfe{{\mathbf e}}
\def\bff{{\mathbf f}}

\def\bfx{{\mathbf x}}
\def\bfy{{\mathbf y}}

\def\calB{{\mathcal B}}

\def\calD{{\mathcal D}}

\def\calJ{{\mathcal J}}

\def\calM{{\mathcal M}}

\def\C{{\mathbb C}}\def\F{{\mathbb F}}
\def\R{{\mathbb R}}
\def\Z{{\mathbb Z}}\def\Q{{\mathbb Q}}

\def\grS{{\mathfrak S}}

\def\alp{{\alpha}} \def\bfalp{{\boldsymbol \alpha}}
  
\def\gam{{\gamma}} 

\def\Gam{{\Gamma}}
\def\del{{\delta}} \def\Del{{\Delta}}
\def\deltil{{\widetilde \delta}}

\def\tet{{\theta}}

\def\deltil{{\tilde \del}}

\def\psitil{{\widetilde \psi}}

 \def\Ome{{\Omega}}

\def\eps{\varepsilon}

\def\atil{{\tilde a}}
\def\gtil{{\tilde g}}

\def\rank{{\rm rank}}

\def\adj{{\rm adj}}
\DeclareMathOperator{\Sing}{Sing}
\DeclareMathOperator{\id}{id}

\newenvironment{blue}{\color{blue}}{}

\specialcomment{comblue}{\begin{blue}}{\end{blue}}
\excludecomment{comblue}

\begin{document}
\title[A variant of Weyl's inequality for systems of forms]{A variant of Weyl's inequality for systems of forms and applications}
\author[Damaris Schindler]{Damaris Schindler}
\address{Hausdorff Center for Mathematics, Endenicher Allee 62-64, 53115 Bonn, Germany}
\email{damaris.schindler@hcm.uni-bonn.de}


\begin{abstract}
We give a variant of Weyl's inequality for systems of forms together with applications. First we use this to give a different formulation of a theorem of B. J. Birch on forms in many variables. More precisely, we show that the dimension of the locus $V^*$ introduced in this work can be replaced by the maximal dimension of the singular loci of forms in the linear system of the given forms. In some cases this improves on the aforementioned theorem of Birch.\par
We say that a system of forms is a Hardy-Littlewood system if the number of integer points on the corresponding variety restricted to a box satisfies the asymptotic behaviour predicted by the classical circle method. 
As a second application, we improve on a theorem of W. M. Schmidt which states that a system of homogeneous forms of same degree is a Hardy-Littlewood system as soon as the so called $h$-invariant of the system is sufficiently large. 
In this direction we generalise previous improvements of R. Dietmann on systems of quadratic and cubic forms to systems of forms of general degree.
\end{abstract}

\maketitle

\excludecomment{com}

\section{Introduction}

We consider a system of homogeneous forms $f_i(x_1,\ldots, x_n)\in \Z[x_1,\ldots, x_n]$ of degree $d$. For convenience we write $\bfx = (x_1,\ldots, x_n)$ and ask for the number of integer solutions to the system of Diophantine equations given by 
\begin{equation*}
f_i(\bfx)=0,\quad 1\leq i\leq r.
\end{equation*}
More precisely, we fix some box $\calB\subset \R^n$ which is contained in the unit box and we let $P\geq 1$ be some real parameter. Then we define the counting function
\begin{equation*}
N(P)=\sharp\{ \bfx\in \Z^n: \bfx\in P\calB,\ f_i(\bfx)=0,\ 1\leq i\leq r\}.
\end{equation*}
This counting function has received a lot of attention so far and is a central object of investigation in number theory. If the number of variables $n$ is relatively large compared to the number of equations and the degree $d$, then the Hardy-Littlewood circle method has proved to be a valuable tool in obtaining asymptotic formulas for the counting function $N(P)$.\par
A very general result in this direction has been obtained by Birch in \cite{Bir62}. He introduces a locus called $V^*$ which is the affine variety given by
\begin{equation*}
\rank \left(\frac{\partial f_i(\bfx)}{\partial x_j}\right)_{\substack{1\leq i\leq r\\ 1\leq j\leq n}} <r.
\end{equation*}
In his work \cite{Bir62} Birch provides an asymptotic formula for $N(P)$ as soon as
\begin{equation*}
n-\dim V^*> r(r+1)(d-1)2^{d-1}.
\end{equation*}
A main ingredient in most applications of the circle method, as for example the one in \cite{Bir62}, is a form of Weyl's inequality. In this paper we present a variant of Weyl's inequality  for systems of forms, and give two applications of our new form of Weyl's inequality.\par
First this allows us to replace the dimension of the locus $V^*$ in Birch's theorem on system of forms by a quantity which appears to be more natural in this context. For some integer vector $\bfb\in \Z^r$ we let $f_\bfb= b_1f_1+\ldots +b_rf_r$ be the form in the pencil of $f_1,\ldots, f_r$ associated to $\bfb$. For any homogeneous form $g$ we write $\Sing (g)$ for the singular locus (in affine space) of the form $g=0$. We can now state a variant of Birch's theorem on forms in many variables as follows.

\begin{theorem}\label{thm1}
Assume that
\begin{equation*}
n-\max_{\bfb\in \Z^r\setminus\{0\}}( \dim \Sing (f_\bfb) )> r(r+1) (d-1)2^{d-1}.
\end{equation*}
Then we have the asymptotic formula
\begin{equation}\label{eqnHL}
N(P)= \grS \calJ P^{n-rd} +O(P^{n-rd-\del}),
\end{equation}
for some $\del >0$. Here $\grS$ and $\calJ$ are the singular series and singular integral.\end{theorem}

This is essentially the main theorem of Birch's work \cite{Bir62} where the quantity $\dim V^*$ is replace by $\max_{\bfb\in \Z^r\setminus\{0\}} (\dim \Sing (f_\bfb))$. In other words we can now describe the singularity of the system of forms $f_i$, $1\leq i\leq r$ by the maximal dimension of the singular loci of forms in the pencil. To our knowledge, there is in contrast no satisfactory geometric interpretation for the locus $V^*$ available.\par
Furthermore, we point out that for any non-trivial form $f_\bfb$ in the pencil, the dimension of the singular locus $\dim \Sing (f_\bfb)$ is always bounded by $\dim V^*$. Indeed, the singular locus of the form $f_\bfb$ is given by
\begin{equation*}
b_1\frac{\partial f_1}{\partial x_i}(\bfx)+\ldots + b_r \frac{\partial f_r}{\partial x_i}(\bfx)=0,\quad 1\leq i\leq n.
\end{equation*}
If some vector $\bfx$ is contained in $\Sing (f_\bfb)$, then these relations imply that the rank of the matrix $(\frac{\partial f_i}{\partial x_j})$ can be at most $r$. This shows that $\Sing (f_\bfb)\subset V^*$, and $\dim \Sing(f_\bfb)\leq \dim V^*$ for any non-zero vector $\bfb$.\par
Hence Theorem \ref{thm1} formally implies Birch's theorem in \cite{Bir62}. Furthermore, there are examples of systems where Theorem \ref{thm1} is stronger than Birch's main theorem in \cite{Bir62}. For simplicity of notation let
\begin{equation*}
u=u(\bff):= \max_{\bfb\in\Z^r\setminus\{0\}}(\dim \Sing (f_\bfb)).
\end{equation*}
Let $k\geq r-1$ be some integer and consider the system of quadratic forms
\begin{equation*}
Q_i(\bfx,\bfy)= \sum_{j=1}^k x_j y_{ji},\quad 1\leq i\leq r,
\end{equation*}
in the $k(r+1)$ variables $x_j$ for $1\leq j \leq k$ and $y_{ji}$ for $1\leq i\leq r$ and $1\leq j\leq k$. A short computation reveals that 
\begin{equation*}
\dim V^*=k(r-1)+r-1= u+r-1.
\end{equation*}
Note also that once we choose $k$ sufficiently large, Theorem \ref{thm1} is indeed applicable. On the other hand these examples are essentially sharp. If we work over the complex numbers then we have
\begin{equation*}
V^*= \cup_{\bfb\in\C^r\setminus\{0\}} \Sing (f_\bfb),
\end{equation*}
and this leads to the bound $\dim V^*\leq u_\C +r-1$, where
\begin{equation*}
u_\C:=  \max_{\bfb\in\C^r\setminus\{0\}}(\dim \Sing (f_\bfb)).
\end{equation*}
Since published in 1962, Birch's work \cite{Bir62} has received a lot of attention and has been generalised in multiple directions. It seems natural to expect that our observation and new formulation of the main result in Theorem \ref{thm1} can in an analogous way be transferred to most of these generalisations and developments. Some examples to mention are work of Brandes \cite{BraA13} on forms representing forms and the vanishing of forms on linear subspaces.
Furthermore, the analogue of the locus $V^*$ in work of Skinner \cite{Ski97}, which generalises Birch's theorem on forms in many variables to the number field situation, and work of the author \cite{bihomforms} on bihomogeneous forms, could very likely be replaced by a non-singularity condition on forms of the linear system. Another result and application in this direction is a paper of Lee \cite{LeeAlan} on a generalisation to function fields $\F_q[t]$. 

As a second application of our new form of Weyl's inequality for systems of forms, we can strengthen a theorem of Schmidt \cite{Schmidt85}, which provides an asymptotic formula for the counting function $N(P)$ as soon as a so-called $h$-invariant of the system is sufficiently large. As a special case of this we recover the results of Dietmann's work \cite{DieA12} on systems of quadratic and cubic forms.\par
For a homogeneous form $f(\bfx)\in \Q[\bfx]$ we define the $h$-invariant of $f$ to be the least integer $h$ such that $f$ can be written in the form
\begin{equation*}
f(\bfx)=\sum_{i=1}^h g_i(\bfx) g_i'(\bfx),
\end{equation*}
with forms $g_i(\bfx)$ and $g_i'(\bfx)$ of positive degree with rational coefficients. For a system $\bff$ of homogeneous forms $f_i(\bfx),$ $1\leq i\leq r$, of degree $d$, we define the $h$-invariant $h(\bff)$ to be the minimum of the $h$-invariant of any form in the rational linear system of the forms, i.e. we set $h(\bff)=\min_{\bfb\in\Z^r\setminus\{0\}}h(f_\bfb)$.\par
We say that a system of forms $\bff$ of degree $d$ is a Hardy-Littlewood system if the conclusion on the asymptotic formula for the counting function $N(P)$ as in equation (\ref{eqnHL}) in Theorem \ref{thm1} holds. If the $h$-invariant of a system of homogeneous forms of the same degree is sufficiently large, then Schmidt proves in his work \cite{Schmidt85} that $\bff$ is a Hardy-Littlewood system. As we shall indicate in section 3, his results easily imply the following theorem, which we state here for convenience.

\begin{theorem}\label{thm2a}[Schmidt, 1985, see \cite{Schmidt85}]
There exists a function $\phi(d)$ with the following property. If the system $\bff$ of homogeneous forms of degree $d>1$ has a $h$-invariant which is bounded below by
\begin{equation*}
h(\bff)> \phi(d) (r(r+1)(d-1)2^{d-1}+(d-1)r(r-1)),
\end{equation*}
then the system $\bff$ is a Hardy-Littlewood system.
Furthermore, one has $\phi(2)=\phi(3)=1$, $\phi(4)=3$, $\phi(5)=13$ and $\phi(d)<(\log 2)^{-d} d!$ in general. 
\end{theorem}

Note that the function $\phi(d)$ is exactly the function occurring in Proposition $III_C$ in Schmidt's work \cite{Schmidt85}.\par
Our new form of Weyl's inequality impoves on this theorem in the following way.

\begin{theorem}\label{thm2}
Let $\phi(d)$ be the function as in Theorem \ref{thm2a}. If the system $\bff$ of homogeneous forms of degree $d>1$ has a $h$-invariant which is bounded below by
\begin{equation*}
h(\bff)> \phi(d) r(r+1)(d-1)2^{d-1},
\end{equation*}
then $\bff$ is a Hardy-Littlewood system.
In the case $d=2$ one may replace the condition on the $h$-invariant of the system by the assumption that the rank of each form in the rational linear system of the quadratic forms is bounded below by $2r(r+1)$.
\end{theorem}

The special cases of degree $d=2$ and $d=3$ in Theorem \ref{thm2} reduce to Theorem 1 and Theorem 2 in Dietmann's paper \cite{DieA12}. In the quadratic case Dietmann improves on previous results of Schmidt in \cite{Schmidt80} in reducing the lower bound in the rank condition from $2r^2+3r$ to only $2r^2+2r$, and in the cubic case he reduces the lower bound on the $h$-invariant from $10r^2+6r$ (see Schmidt's paper \cite{Schmidt82}) to $8r^2+8r$. In fact, our new form of Weyl's inequality takes up the main idea in Dietmann's work \cite{DieA12}.\par
As Dietmann points out in \cite{DieA12}, the $h$-invariant can in some ways be seen as a generalisation of the rank of a quadratic form to higher degree forms. By diagonalising a quadratic form one sees that its $h$-invariant is bounded by its rank. However, we note that these two notions do not coincide for the case of quadratic forms, as examples built up from forms like $x_1^2-x_2^2= (x_1+x_2)(x_1-x_2)$ show. Hence we need to formulate the case $d=2$ in Theorem \ref{thm2} separately in order to obtain the full strength of the theorem in this case.\par
As another example we consider the case of systems of forms $\bff$ of degree $d=4$. In this case one has $\phi(4)=3$ and Theorem \ref{thm2} implies that the expected asymptotic formula for $N(P)$ holds as soon as 
\begin{equation*}
h(\bff)> 3r(r+1)\cdot3\cdot2^{3}= 9\cdot  (8r^2+8r).
\end{equation*}
Schmidt obtains the same result in his paper \cite{Schmidt85} (see Theorem \ref{thm2a} above) under the stronger condition
\begin{equation}
h(\bff) > \phi(4) (r(r+1)(d-1)2^{d-1}+(d-1)r(r-1))= 9(9r^2+7r).
\end{equation}
We finally remark that if the system of forms $f_i(\bfx)$, $1\leq i\leq r$, in Theorem \ref{thm1} or Theorem \ref{thm2} forms a complete intersection, and if there exist non-singular real and $p$-adic points on the variety $X$ given by these forms, then the singular series $\grS$ and the singular integral $\calJ$ are both positive. In particular, this implies the existence of rational points on the variety $X$ as soon as there are non-singular solutions at every place of $\Q$ including infinity.\par

The structure of this paper is as follows. We recall a version of Weyl's inequality from \cite{Bir62} in the next section and present in Lemma \ref{birlem2} our new variant of Weyl's inequality for systems of forms. We use this in the last section to deduce Theorem \ref{thm1} and Theorem \ref{thm2}, and we explain the improvements of Theorem \ref{thm2} compared to Theorem \ref{thm2a}.\par
\textbf{Acknowledgements.} The author would like to thank Prof. T. D. Browning for comments on an earlier version of this paper and Prof. P. Salberger for helpful discussions.

\section{A variant of Weyl's inequality}

For some $n$-dimensional box $\calB$, some real vector $\bfalp=(\alp_1,\ldots, \alp_r)$ and some large real number $P$ we define the exponential sum
\begin{equation*}
S(\bfalp)= \sum_{\bfx\in P\calB\cap \Z^n}e\left(\sum_{i=1}^r \alp_if_i(\bfx)\right).
\end{equation*}
If $f(\bfx)$ is some homogeneous form of degree $d$, then we let $\Gam_f (\bfx^{(1)},\ldots, \bfx^{(d)})$ be its unique associated symmetric multilinear form satisfying $\Gam_f(\bfx,\ldots,\bfx)=d! f(\bfx)$. Moreover, if $f_i(\bfx)$, $1\leq i\leq r$, form a system of homogeneous forms of degree $d$ as before, then we let $\Gam_i(\bfx^{(1)},\ldots, \bfx^{(d)})$, $1\leq i\leq r$, be the associated multilinear forms. We introduce the sup-norm $|\bfx|=\max_{1\leq i\leq n} |x_i|$ on the vector space $\R^n$, and write $\Vert \gam\Vert= \min_{y\in\Z}|\gam-y|$ for the least distance of a real number $\gam$ to an integer. 
Furthermore, we write here and in the following $\bfe_j$ for the $j$-th unit vector in $n$-dimensional affine space. Then we let $N(P^\xi;P^{-\eta};\bfalp)$ be the number of integer vectors $\bfx^{(2)},\ldots, \bfx^{(d)}$ with $|\bfx^{(2)}|,\ldots, |\bfx^{(d)}|\leq P^\xi$ and
\begin{equation*}
\left\Vert \sum_{i=1}^r \alp_i\Gam_i (\bfe_j,\bfx^{(2)},\ldots, \bfx^{(d)})\right\Vert <P^{-\eta},\quad 1\leq j\leq n.
\end{equation*}

We start our considerations with recalling Lemma 2.4 from Birch's work \cite{Bir62}.
\begin{lemma}[Lemma 2.4 in \cite{Bir62}]\label{bir1}
For fixed $0<\tet\leq 1$ one of the following alternatives hold.\\
i) $|S(\bfalp)|< P^{n-k}$, or\\
ii) $N(P^\tet; P^{-d+(d-1)\tet};\bfalp)\gg P^{(d-1)n\tet - 2^{d-1}k-\eps}$, for any $\eps>0$.
\end{lemma}

The main idea is to treat the condition ii) differently than in Birch's work \cite{Bir62}, following a similar idea as taken up in the paper \cite{DieA12}. Before stating our new version of Weyl's inequality, we need to introduce the $g$-invariant of a homogeneous form.\par
We define $\calM_f$ to be the variety in affine $(d-1)n$-space given by
\begin{equation*}
\Gam_f(\bfe_j,\bfx^{(2)},\ldots, \bfx^{(d)})=0,\quad 1\leq j\leq n,
\end{equation*}
and write $\calM_f(P)$ for the number of integer points on $\calM_f$ with coordinates all bounded by $P$. Then we define the $g$-invariant $g(f)$ of the form $f$ to be the largest real number such that 
\begin{equation*}
\calM_f (P)\ll P^{(d-1)n-g(f)+\eps},
\end{equation*}
holds for all $\eps >0$. Note that this number $g(f)$ coincides with the $g$-invariant that Schmidt associates to a single form $f$ in \cite{Schmidt85}.\par

\begin{lemma}\label{birlem2}
Let $\gtil=\inf_{\bfb\in \Z^r\setminus \{0\}}g (f_\bfb)$ and let $0<\tet\leq 1$ be fixed. Then we either have the bound\\
i) $|S(\bfalp)|<P^{n-2^{-d+1}\gtil\tet+\eps}$, or\\
ii) (major arc approximation for $\bfalp$ with respect to the parameter $\tet$) there exist natural numbers $a_1,\ldots, a_r$ and $1\leq q\ll P^{r(d-1)\tet}$ with $\gcd (q,a_1,\ldots, a_r)=1$ and 
\begin{equation*}
|q\alp_i-a_i|\ll P^{-d+r(d-1)\tet},\quad 1\leq i\leq r.
\end{equation*}
\end{lemma}

\begin{proof}
Let the notation be as in Lemma \ref{bir1} and assume that alternative (ii) in Lemma \ref{bir1} holds.\par
We consider the matrix $\psi$ of size $r\times (nN(P^\tet; P^{-d+(d-1)\tet};\bfalp))$ with entries $\Gam_i(\bfe_j,\bfx^{(2)},\ldots, \bfx^{(d)})$ in the $i$th row. The columns are indexed by $(j,\bfx^{(2)},\ldots, \bfx^{(d)})$, where $1\leq j\leq n$ and $\bfx^{(2)},\ldots, \bfx^{(d)}$ run through all tuples of integer vectors counted by $N(P^\tet; P^{-d+(d-1)\tet};\bfalp)$. We distinguish two cases.\par
Case (a): Assume that $\rank (\psi) =r$. Then there is a $r\times r$-minor $\psitil$ of full rank, which we say is given by
\begin{equation*}
\psitil = (\psitil_{i,l})_{1\leq i,l\leq r}=(\Gam_i(\bfe_{j_l},\bfx^{(2)}_l,\ldots, \bfx_l^{(d)}))_{1\leq i,l\leq r}.
\end{equation*}
In particular, we have $\Vert \sum_{i=1}^r \alp_i \psitil_{i,l}\Vert < P^{-d+(d-1)\tet}$, for all $1\leq l\leq r$. Hence, we can write
\begin{equation*}
\sum_{i=1}^r \alp_i\psitil_{i,l}= \atil_l+\deltil_l,
\end{equation*}
with integers $\atil_l$ and real numbers $\deltil_l$ with $|\deltil_l|<P^{-d+(d-1)\tet}$. Let $\psitil^\adj$ be the adjoint matrix to $\psitil$, which satisfies $\psitil^\adj \psitil = (\det \psitil) \id$, and let $q=\det \psitil$. Since $\psitil$ was assumed to be of rank $r$, its determinant $q$ is non-zero. Furthermore we note that $|q|\ll P^{r\tet (d-1)}$. Now we can use the adjoint matrix $\psitil^\adj$ to find a good approximation for $\bfalp$ by rational numbers with small denomiator. Indeed, we have
\begin{align*}
\left|q\alp_i-\sum_{l=1}^r\psitil^\adj_{i,l}\atil_l\right| &\leq \sum_{l=1}^r |\psitil_{i,l}^\adj||\deltil_l| \\ &\ll P^{\tet (r-1)(d-1)}P^{-d+(d-1)\tet}.
\end{align*}
We set $a_i=\sum_{l=1}^r\psitil_{i,l}^\adj \atil_l$. After removing common factors of $q$ and the integers $a_i$, we obtain integers $q,a_1,\ldots, a_r$ with $\gcd (q,a_1,\ldots, a_r)=1$ and $1\leq q\ll P^{r(d-1)\tet}$, such that 
\begin{equation*}
|q\alp_i-a_i|\ll P^{-d+r(d-1)\tet},\quad 1\leq i\leq r.
\end{equation*}
In this case the conclusion (ii) of the Lemma holds.\par
Case (b): Assume that $\rank (\psi) <r$. Then the $r$ rows of $\psi$ are linearly dependent over $\Q$, and thus there exist integers $b_1,\ldots, b_r\in \Z$, not all zero, such that 
\begin{equation*}
\sum_{i=1}^r b_i \Gam_i(\bfe_j,\bfx^{(2)},\ldots, \bfx^{(d)})=0,
\end{equation*}
for all $1\leq j\leq n$ and for all tuples $\bfx^{(2)},\ldots, \bfx^{(d)}$ counted by $N(P^\tet; P^{-d+(d-1)\tet};\bfalp)$.\par
We note that 
\begin{equation*}
\sum_{i=1}^r b_i\Gam_i(\bfx^{(1)},\ldots, \bfx^{(d)})= \Gam_{f_\bfb}(\bfx^{(1)},\ldots, \bfx^{(d)})
\end{equation*}
is the multilinearform associated to the form $f_\bfb(\bfx)= \sum_{i=1}^r b_if_i(\bfx)$. We recall the definition of the variety $\calM_f$ stated before this lemma, and deduce from the lower bound on $N(P^\tet; P^{-d+(d-1)\tet};\bfalp)$, that we have
\begin{equation}\label{bireqn2}
\calM_{f_\bfb}(P^\tet)\gg P^{(d-1)n\tet - 2^{d-1}k-\eps},
\end{equation}
for any $\eps >0$.
By definition of the $g$-invariant $g(f_\bfb)$ we see that equation (\ref{bireqn2}) implies that
\begin{equation*}
(d-1)n-2^{d-1}k/\tet-\eps\leq (d-1)n-g(f_\bfb)+\eps,
\end{equation*}
for any $\eps >0$. Hence we have $2^{-d+1}g(f_\bfb)\tet \leq k$ for some $\bfb\in\Z^r\setminus\{0\}$, and the first alternative of the lemma holds.
\end{proof}

\section{Applications}

The main goal of this section is to prove Theorem \ref{thm1} and Theorem \ref{thm2}. Before we show how Lemma \ref{birlem2} implies Theorem \ref{thm1} and Theorem \ref{thm2} we recall some very general results from Schmidt's work \cite{Schmidt85}, which simplify our following arguments.\par
For this we first recall the Hypothesis on $\bff$ introduced in section 4 in \cite{Schmidt85} for the case of forms of same degree.

\begin{hypothesis}[Hypothesis on $\bff$ with parameter $\Ome$]\label{hyp}
Let $\calB$ be some box and $\Del >0$ and assume that $P$ is sufficiently large depending on the system $\bff$, the parameter $\Ome$, the box $\calB$ and $\Del$. Then one either has the upper bound\\
i) $|S(\bfalp)|\leq P^{n-\Del\Ome}$, or\\
ii) there are natural numbers $q\leq P^\Del$ and $a_1,\ldots, a_r$ such that
\begin{equation*}
|q\alp_i-a_i|\leq P^{-d+\Del},\quad 1\leq i\leq r.
\end{equation*}
\end{hypothesis}

In his work \cite{Schmidt85} Schmidt shows that this hypothesis is enough to verify that the Hardy-Littlewood circle method can be applied to the counting function $N(P)$ related to the system of equations $\bff$. One of his main results is the following, which we only state for the special case of forms of same degree, since this is all we use in this paper.

\begin{theorem}[Proposition I in \cite{Schmidt85}, second part]\label{propI}
Suppose the system $\bff$ satisfies Hypothesis \ref{hyp} with respect to some parameter
\begin{equation*}
\Ome > r+1.
\end{equation*}
Then $\bff$ is a Hardy-Littlewood system.
\end{theorem}

In combination with our new version of Weyl's inequality in Lemma \ref{birlem2} we obtain the following useful Corollary.

\begin{corollary}\label{cor1}
Assume that $\bff$ is a system of homogeneous forms of same degree with
\begin{equation*}
\inf_{\bfb\in \Z^r\setminus \{0\}}g (f_\bfb) > r(r+1)(d-1)2^{d-1}.
\end{equation*}
Then the asymptotic formula for $N(P)$ as predicted by the circle method holds, i.e. $\bff$ is a Hardy-Littlewood system.
\end{corollary}

\begin{proof}
First we note that Lemma \ref{birlem2} implies Hypothesis \ref{hyp} with respect to any parameter 
\begin{equation*}
\Ome < 2^{-d+1} r^{-1}(d-1)^{-1}\gtil.
\end{equation*}
This is clear by the formulation of Lemma \ref{birlem2} for the range $0<\Del\leq r(d-1)$ if we set $\Del = r(d-1)\tet$. In the case where $\Del >r(d-1)$ alternative (ii) in Hypothesis \ref{hyp} is automatically satisfied by Dirichlet's approximation principle.\par
Now Proposition I in \cite{Schmidt85} applies as stated in Theorem \ref{propI}, which completes the proof of the Corollary.
\end{proof}

Next we relate the $g$-invariant of a homogeneous form to the dimension of its singular locus, and thereby establish the new formulation of Birch's theorem on forms in many variables as stated in Theorem \ref{thm1}. 

\begin{proof}[Proof of Theorem \ref{thm1}]
Consider some vector $\bfb\in\Z^r\setminus\{0\}$ and its associated form $f_\bfb$ in the pencil of $f_i(\bfx)$, $1\leq i\leq r$. We note that the intersection of the affine variety $\calM_{f_\bfb}$ with the diagonal $\calD$ given by
\begin{equation*}
\bfx^{(2)}=\ldots = \bfx^{(d)},
\end{equation*}
is isomorphic to the singular locus of the form $f_\bfb$. Hence we obtain by the affine intersection theorem that
\begin{equation*}
\dim \Sing(f_\bfb) = \dim (\calD\cap \calM_{f_\bfb}) \geq \dim \calD + \dim \calM_{f_\bfb}- (d-1)n.
\end{equation*}
This shows that
\begin{equation*}
\dim \Sing (f_\bfb)\geq \dim \calM_{f_\bfb}-(d-1)n+n,
\end{equation*}
which implies that
\begin{equation*}
g(f_\bfb)\geq n-\dim \Sing (f_\bfb).
\end{equation*}
Taking the the infimum over all non-zero integer tuples $\bfb$ we obtain
\begin{equation*}
\inf_{\bfb\in\Z^r\setminus\{0\}}g(f_\bfb)\geq n-\max_{\bfb\in\Z^r\setminus\{0\}}(\dim \Sing(f_\bfb)),
\end{equation*}
and hence Lemma \ref{birlem2} holds with $\gtil$ replaced by $n-\max_{\bfb\in\Z^r\setminus\{0\}}(\dim \Sing(f_\bfb))$. This shows that in Lemma 4.3 in Birch's work \cite{Bir62}, the quantity $K$ which is defined in his setting as 
\begin{equation*}
2^{d-1}K= n-\dim V^*,
\end{equation*}
can be replaced by
\begin{equation*}
2^{d-1}K= n-\max_{\bfb\in\Z^r\setminus\{0\}}(\dim \Sing(f_\bfb)).
\end{equation*}
Now Theorem \ref{thm1} follows identically as the main theorem in Birch's paper \cite{Bir62}. Alternatively we can apply Corollary \ref{cor1} to obtain the desired result.
\end{proof}

Next we turn towards the proof of Theorem \ref{thm2} which improves on the previous known results in Theorem \ref{thm2a}. However, since Theorem \ref{thm2a} is not contained in this formulation in the paper \cite{Schmidt85}, we first give a short deduction of it from the results of \cite{Schmidt85}. 
Indeed, Schmidt concludes in his remark after Proposition $II_0$ that the expected asymptotic formula in Theorem \ref{thm2a} for $N(P)$ holds as soon as a so called $g$-invariant $g(\bff)$ of the system $\bff$ is larger than
\begin{equation}\label{eqn4a}
g(\bff)>2^{d-1}(d-1)r(r+1).
\end{equation}
His Corollary after Proposition III states that there is the relation
\begin{equation}\label{eqn4b}
h(\bff)\leq \phi(d) (g(\bff)+(d-1)r(r-1)).
\end{equation}
Hence the condition 
\begin{equation*}
h(\bff)> \phi(d) (r(r+1)(d-1)2^{d-1}+(d-1)r(r-1)),
\end{equation*}
in Theorem \ref{thm2a} implies that (\ref{eqn4a}) holds and thus the conclusion of Theorem \ref{thm2a} follows.\par
The main difference in the use of our new version of Weyl's inequality in comparison to Schmidt's work is that we can state everything in terms of the $g$-invariant of a single form. In his work \cite{Schmidt85} Schmidt uses a form of Weyl's inequality where the infimum of all $g$-invariants of the elements of the rational linear system is replaced by his so called $g$-invariant $g(\bff)$ of the whole system. This is in complete analogy with the replacement of the locus $V^*$ in Birch's work \cite{Bir62} by the maximal dimension of the singular loci of elements in the rational linear system as in Theorem \ref{thm1}. 

\begin{proof}[Proof of Theorem \ref{thm2}]
For a single form $f$, equation (17.2) in \cite{Schmidt85} implies that
\begin{equation}\label{eqn4d}
h(f)\leq \phi(d) g(f),
\end{equation}
which should be compared to the relation (\ref{eqn4b}) for systems of forms. For a single form we do not need the term $\phi(d)(d-1)r(r-1)$, which is present for the relation referring to the whole system of forms in equation (\ref{eqn4b}). This explains our improvement of Theorem \ref{thm2} compared to Theorem \ref{thm2a}.\par
Assume now that the assumptions of Theorem \ref{thm2} are satisfied, i.e. 
\begin{equation*}
h(\bff)> \phi(d) r(r+1)(d-1)2^{d-1}.
\end{equation*}
Recall that we have defined $h(\bff)=\min_{\bfb\in\Z^r\setminus\{0\}}h(f_\bfb)$. Hence we obtain together with equation (\ref{eqn4d}) the relation
\begin{equation*}
\inf_{\bfb\in\Z^r\setminus\{0\}} g(f_\bfb)> r(r+1)(d-1)2^{d-1}.
\end{equation*}
Now we apply Corollary \ref{cor1} to complete the proof of Theorem \ref{thm2} for the case of degree $d\geq 3$.\par
For the case of systems of quadratic forms we note that the $g$-invariant of a single quadratic form $f$ is bounded below by its rank. Indeed, let some quadratic form $f$ be given by some $n\times n$-matrix $A$. Then the variety $\calM_f$ is given by the system of linear equations $A\bfx=0$, and we deduce that
\begin{equation*}
\calM_f(P)\ll P^{n-\rank (A)},
\end{equation*}
which shows that $g(f)\geq \rank (A)$. Now we apply Corollary \ref{cor1} as in the case $d\geq 3$.
\end{proof}

\bibliographystyle{amsbracket}
\providecommand{\bysame}{\leavevmode\hbox to3em{\hrulefill}\thinspace}

\end{document}